\documentclass[12pt]{amsart}

\usepackage{amssymb,enumerate,amsthm,yfonts,mathrsfs}
\usepackage{amsmath}
\usepackage{bbm}           % used for blackboard 1
\usepackage{bm} 
\usepackage{eso-pic,graphicx}
\usepackage{tikz}

\usepackage{mathtools}
\usepackage{rotating}

\usepackage[colorlinks=true, pdfstartview=FitV, linkcolor=blue, citecolor=blue, urlcolor=blue]{hyperref}

\calclayout

\makeatletter
\def\Ddots{\mathinner{\mkern1mu\raise\p@
\vbox{\kern7\p@\hbox{.}}\mkern2mu
\raise4\p@\hbox{.}\mkern2mu\raise7\p@\hbox{.}\mkern1mu}}
\makeatother

\def\Xint#1{\mathchoice
{\XXint\displaystyle\textstyle{#1}}%
{\XXint\textstyle\scriptstyle{#1}}%
{\XXint\scriptstyle\scriptscriptstyle{#1}}%
{\XXint\scriptscriptstyle\scriptscriptstyle{#1}}%
\!\int}
\def\XXint#1#2#3{{\setbox0=\hbox{$#1{#2#3}{\int}$}
\vcenter{\hbox{$#2#3$}}\kern-.5\wd0}}

\def\dashint{\Xint-}

\newtheorem{theorem}{Theorem}[section]

\newtheorem{corollary}[theorem]{Corollary}

\newtheorem{lemma}[theorem]{Lemma}

\newtheorem{remark}[theorem]{Remark}

\numberwithin{equation}{section}
 \allowdisplaybreaks

\def\R{\mathbb R}

\def\Z{\mathbb Z}

\def\bey{\begin{eqnarray*}}
\def\eey{\end{eqnarray*}}

\DeclareMathOperator{\supp}{supp}

\def\D{{\mathscr D}}

\newcommand\de{\delta}
\newcommand\s{\sigma}

\title[Logarithmic bump conditions on spaces of homogeneous type]{Logarithmic bump
  conditions for Calder\'on-Zygmund Operators on spaces of homogeneous type}

\author{Theresa C. Anderson, David Cruz-Uribe, SFO and Kabe Moen }

\address{Theresa C. Anderson, Department of Mathematics, Brown
  University, Providence, RI 02912, USA}
\email{theresa\_anderson@brown.edu}

\address{David Cruz-Uribe, SFO, Department of Mathematics, Trinity
  College, Hartford, CT 06106, USA}
\email{david.cruzuribe@trincoll.edu}

\address{Kabe Moen, Department of Mathematics, University of Alabama,
  Tuscaloosa, AL 35487, USA}
\email{kabe.moen@ua.edu}

\thanks{The first author is supported by an NSF graduate fellowship.
The second author is supported by the Stewart-Dorwart faculty
  development fund at Trinity College and by
  grant MTM2009-08934 from the Spanish Ministry of Science and
  Innovation.  The third author is supported by NSF Grant 1201504}



\date{August 2, 2013}

\begin{document}
\maketitle

\begin{abstract}
  We establish two-weight norm inequalities for singular integral
  operators defined on spaces of homogeneous type.  We do so first
  when the weights satisfy a double bump condition and then when the
  weights satisfy separated logarithmic bump conditions.  Our results
  generalize recent work on the Euclidean case, but our proofs are
  simpler even in this setting. The other interesting feature of our
  approach is that we are able to prove the separated bump results
  (which always imply the corresponding double bump results) as a
  consequence of the double bump theorem.
\end{abstract}

\section{Introduction and main results}
\label{sec:introduction}

Given a Calder\'on-Zygmund singular integral $T$, the problem of finding
sufficient conditions on a pair of weights $(u,\sigma)$ such that the
two-weight norm inequality
\begin{equation}\label{eqn:intro-norm}
 \int_{\R^n} |T(f\sigma)(x)|^pu(x)\,dx \leq C\int_{\R^n}
|f(x)|^p\sigma(x)\,dx
\end{equation}
holds dates back to the 1970s. Significant progress has only
been made in the past twenty years:  for a brief history,
see~\cite[Chapter~1]{MR2797562}.  One approach to this problem is to
use the so-called $A_p$-bump conditions introduced by
P\'erez~\cite{perez94,perez95}.  It was conjectured that a sufficient
condition for~\eqref{eqn:intro-norm} to hold is that the pair
$(u,\sigma)$ satisfies
\begin{equation} \label{eqn:double-bump}
\sup_Q\|u^{1/p}\|_{A,Q} \|\sigma^{1/p'}\|_{B,Q} < \infty,
\end{equation}
where the supremum is taken over all cubes in $\R^n$, $A$ and $B$ are
Young functions that satisfy the growth conditions $\bar{A} \in
B_{p'}$ and $\bar{B}\in B_p$, and $\|\cdot\|$ is a normalized Orlicz
norm.  (For precise definitions, see 
below.)  This problem proved quite difficult, and a number of partial
results were proved \cite{cruz-uribe-martell-perezP,dcu-martell-perez,
  cruz-uribe-perez00b,cruz-uribe-perez02} before the full result
was proved by Lerner~\cite{Lern2012} and by Nazarov, Reznikov and
Volberg~\cite{NRVbump} (when $p=2$).  Much of the recent progress on
this problem was due to the close connection with the $A_2$ conjecture
on sharp one-weight norm inequalities for singular
integrals---see~\cite{dcu-martell-perez} for details.

Recently, it was noted~\cite{CRV2012} that while the conjecture was originally stated in
terms of the ``double bump'' condition~\eqref{eqn:double-bump},
it was motivated by the so-called Muckenhoupt-Wheeden conjectures
(see~\cite{CRV2012} and~\cite[Section~9.2]{MR2797562})
and
that implicit in this motivation was a weaker conjecture in terms of a
pair of ``separated bump'' conditions:  $T$
satisfies~\eqref{eqn:intro-norm} if the pair $(u,\sigma)$ satisfy
\begin{equation} \label{eqn:sep-bump}
\sup_Q\|u^{1/p}\|_{A,Q} \|\sigma^{1/p'}\|_{p',Q} < \infty, \qquad
\sup_Q\|u^{1/p}\|_{p,Q} \|\sigma^{1/p'}\|_{B,Q} < \infty.
\end{equation}
In~\cite{CRV2012} this conjecture was proved in the special case
when $A$ and $B$ are ``log bumps'':  i.e.,
$A(t)=t^p\log(e+t)^{p-1+\delta}$ and
$B(t)=t^{p'}\log(e+t)^{p'-1+\delta}$, $\delta>0$.
A simpler proof, one which also gives quantitative estimates on the
constants for separated bumps, was found by Hyt\"onen and
P\'erez~\cite{hytonen-perezP2013}.   The exact value of the constants
is important, since 
Hyt\"onen~\cite{hytonen-2012} has shown that if the
sharp constants for the separated bump condition are as conjectured, 
then as an immediate  corollary this result yields a new proof of the
sharp $A_p$-$A_\infty$ bounds for singular
integrals~\cite{hytonen-lacey-IUMJ,hytonen-perez-analPDE}. 

\begin{remark}
  It has generally been accepted that the separated bump condition is
  weaker than the double bump condition, but no explicit pair $(u,v)$
  that satisfies~\eqref{eqn:sep-bump}  but not \eqref{eqn:double-bump}
  for a given pair of Young functions $A,\,B$ has
  appeared in the literature.  We rectify this by constructing an
  example in Section~\ref{section:example} below.
\end{remark}

\medskip

The goal of this paper is to extend the double bump and separated bump
results discussed above to the case of singular integrals on spaces of
homogeneous type.  These spaces are of interest since they often arise
in applications: see for example~\cite{MR1104656,MR0499948,MR0447954,
  MR2033231,MR1962076}.  Many of the tools of classical harmonic
analysis on Euclidean spaces generalize to this setting; nevertheless
there are substantive differences and some care must be taken to
insure that proofs still hold.  Our arguments differ extensively from
those in~\cite{CRV2012}:    they have more in common with the approach
taken in~\cite{hytonen-perezP2013}.    Our proof, when restricted to
the Euclidean case is somewhat simpler than theirs, but we do not 
prove the same quantitative estimates on the constants.    A
very interesting feature of our proof is that we are able to prove the
separated results as a consequence of the double bump estimates.

Before we can state our main results we need to make a number of
definitions.  By a space of homogeneous type (hereafter, SHT) we mean
an ordered triple $(X,\rho,\mu)$ where $X$ is a set, $\rho$ is a
quasimetric on $X$, and $\mu$ is a non-negative Borel measure on $X$ that is
doubling:
\[\mu(B_\rho(x_0,2r))\leq C_d\mu(B_\rho(x_0,r)),\] 
where $B_\rho(x_0,r) = \{ x\in X : \rho(x,x_0)<r\}$.  The smallest such
constant  $C_d$ is called the doubling constant of $\mu$.  We also assume that
$\mu$ is non-trivial, i.e., for every ball, $0<\mu(B_\rho(x_0,r))<\infty$.  For further
details, see Christ~\cite{MR1104656} or Coifman and Weiss~\cite{MR0499948}.

\begin{remark}
For brevity, hereafter we will say that a constant depends on $X$ and
write $C(X,\ldots)$ if the constant depends on the triple
$(X,\rho,\mu)$. 
\end{remark}

\smallskip

A function $K:X\times X\setminus\{(x,x)\}\to \R$ is a Calder\'on-Zygmund
kernel if there exist $\eta>0$ and $C<\infty$ such that for all
$x_0\neq y\in X$ and $x\in X$ it satisfies the decay condition:
\begin{equation*}
\label{decay}
|K(x_0,y)|\leq \frac{C}{\mu(B_\rho(x_0,\rho(x_0,y)))}
\end{equation*}
and the smoothness condition: for $\rho(x_0,x)\leq \eta\rho(x_0,y)$,
\begin{equation*}
\label{smoothness}
|K(x,y)-K(x_0,y)|+|K(y,x)-K(y,x_0)|\leq C\left(\frac{\rho(x,x_0)}{\rho(x_0,y)}\right)^{\eta}
\frac{1}{\mu(B_\rho(x_0,\rho(x_0,y)))}.
\end{equation*}
An operator $T$ is associated with a Calder\'on-Zygmund kernel $K$ if for every
$f\in C_0^\infty(X)$, 
\[ Tf(x) = \int_X f(y)K(x,y)\,d\mu(y),  \qquad x\not\in \supp(f). \]
If $T$ is bounded on $L^2(X,\mu)$, then $T$ is referred to as a
Calder\'on-Zygmund operator.

\medskip

The bump conditions discussed above are given in terms of Orlicz
norms.  Here we summarize some of the basic properties we need; for
the general theory of Orlicz spaces, see Rao and Ren \cite{rao-ren91}
or~\cite[Chapter 5]{MR2797562}.
A Young function is a continuous, convex, increasing function $A
:[0,\infty)\to[0,\infty)$ such that $A(0) = 0$ and $A(t)/t\rightarrow
\infty$ as $t\rightarrow \infty$.  It is often convenient to assume
that $A(1)=1$ but this is not strictly necessary.  Note that $A(t)=t$
is not a Young function though $t^p$ is for $p>1$.  However, in many
cases results for Young functions hold in this limiting case.  The
Young functions we are interested in are referred to as log
bumps: $A(t)=t^p\log(e+t)^{p-1+\delta}$, $\delta>0$.

Given two Young functions $A$ and $B$, we write that $A\lesssim B$ if
there exists constants $c,\,t_0>0$ such that $A(t)\leq B(ct)$ for all
$t\geq t_0$.   Note that given any Young function $A$, $t\lesssim A(t)$. 
We will write $A\approx B$ if there exists $c_1,\,c_2,t_0>0$ such that 
$c_1 A(t) \leq B(t) \leq c_2A(t)$ for all $t\geq t_0$.

Given a Young function $A$ and a set $E$ such that $0<\mu(E)<\infty$, define
the Orlicz space norm 
\[\|v\|_{A,E} =
\inf\left\{\lambda>0:\dashint_EA\bigg{(}\frac{|v(x)|}
{\lambda}\bigg{)}d\mu(x)
  \leq 1\right\},\]
where $\dashint_E = \mu(E)^{-1}\int_E$.  If $A(t)=t^p$, $1\leq p<\infty$, then 
\[ \|v\|_{A,E} = \left(\dashint_E  |v|^p\,d\mu\right)^{1/p} = \|v\|_{p,E}. \]
If $A\lesssim B$, then there exists a constant $C$ such that 
$\|f\|_{A,E} \leq C\|f\|_{B,E}$.  

Given a Young function $A$, define $\bar{A}$, the complementary 
function, by
\[ \bar{A}(t) = \sup_{s>0} \{ st - A(s)\}. \]
It can be shown that $\bar{A}$ is also a Young function.  Given $A$,
we have the generalized H\"older's inequality: for any set $E$,
$0<\mu(E)<\infty$,
\[ \dashint_E |f(x)g(x)|\,d\mu(x) \leq 2\|f\|_{A,E}\|g\|_{\bar{A},E}. \]
More generally, given three Young functions $A,\,B,\,C$ such that
\[ B^{-1} C^{-1}(t)\leq c A^{-1}(t), \]
then there exists a constant $K$ such that
\[ \|fg\|_{A,E} \leq K \|f\|_{B,E}\|g\|_{C,E}. \]

Given $1<p<\infty$ we define the class $B_p$: a Young function $A\in
B_p$ if 
\[ [A]_{B_p} = \int_1^{\infty}\frac{A(t)}{t^p}\frac{dt}{t}<\infty.\]
In the special case of log bumps, if $A(t)=t^p\log(e+t)^{p-1+\delta}$
then $\bar{A}(t) \approx t^{p'}\log(e+t)^{-1-(p'-1)\delta}$,
and so $\bar{A}\in B_{p'}$.  

We can now define our bump conditions.  Given  Young
functions $A$ and $B$,  and a pair of
weights $(u,\sigma) $, define
\[ [u,\sigma]_{A,B,p} = \sup_{B_\rho}\|u^{1/p}\|_{A,B_\rho}
\|\sigma^{1/p'}\|_{B,B_\rho}, \]
and
\[
[u,\sigma]_{A,p} = \sup_{B_\rho}\|u^{1/p}\|_{A,B_\rho}
\|\sigma^{1/p'}\|_{p',B_\rho},
\]
where the suprema are taken over all balls $B_\rho$ in $X$.
Note that by symmetry we have that if $B$ is another Young function,
then
\[ [\sigma,u]_{B,p'} = \sup_{B_\rho}
\|u^{1/p}\|_{p,B_\rho}\|\sigma^{1/p'}\|_{B,B_\rho}.
\]

\smallskip

By weights $u$ and $\sigma$ we always mean non-negative measurable
functions on $X$ that are finite almost everywhere and positive on
sets of positive measure.   Many authors assume that weights are
locally integrable; however,  when working with bump conditions this
assumption can be avoided by an approximation argument.  As was shown
in~\cite[Section~7.2]{MR2797562}, we can always assume that $u$ and
$\sigma$ are bounded and bounded away from $0$ on $X$, provided that
in the norm inequality being proved we are working with a function
$f\in \cap_{p>1} L^p(X,\mu)$: for example, $f$ is a bounded function of
compact support.   

\begin{remark}
Since bounded functions of compact support are dense in any weighted
space $L^p(X,u)$, we will hereafter assume that $u$, $\sigma$ and $f$
satisfy these conditions.  Moreover, since $T$ is linear we will also
assume without loss of generality that $f$ is non-negative.
\end{remark}

\bigskip

We can now state our main results.  The first generalizes the double
bump condition to SHT.

\begin{theorem} \label{thm:double}
Given an SHT $(X,\rho,\mu)$, suppose the pair of weights $(u,\sigma)$
satisfies $[u,\sigma]_{A,B,p}<\infty$, where $\bar{A}\in B_{p'}$ and
$\bar{B}\in B_p$.  
Then  a Calder\'on-Zygmund operator $T$ satisfies the strong type
inequality
 \[\|T(f\s)\|_{L^{p}(u)}\leq 
C(T,X)[u,\s]_{A,B,p}[\bar{A}]_{B_{p'}}^{1/p'}[\bar{B}]_{B_p}^{1/p}\|f\|_{L^p(\s)}.\]
\end{theorem}

\medskip

The next two results give separated bump conditions for weak and
strong type inequalities.  

\begin{theorem}
\label{thm:main}
Given an SHT $(X,\rho,\mu)$, suppose the pair of weights $(u,\sigma)$
is such that $[u,\sigma]_{A,p} <\infty$, where $A(t) =
t^plog(e+t)^{p-1+\de}$.
Then  a Calder\'on-Zygmund operator $T$ satisfies the weak type
inequality
 \[\|T(f\s)\|_{L^{p,\infty}(u)}\leq C(T,X,p,\delta)[u,\s]_{A,p}\|f\|_{L^p(\s)}.\]
\end{theorem}

\begin{theorem}
\label{thm:main-strong}
Given an SHT $(X,\rho,\mu)$, suppose the pair of weights $(u,\sigma)$
is such that $[u,\sigma]_{A,p} <\infty$ and $[\sigma,u]_{B,p'}<\infty$, where $A(t) =
t^plog(e+t)^{p-1+\de}$ and $B(t)=t^{p'}\log(e+t)^{p'-1+\delta}$.
Then  a Calder\'on-Zygmund operator $T$ satisfies the strong type
inequality
 \[\|T(f\s)\|_{L^{p}(u)}\leq 
C(T,X,p,\delta)\big([u,\s]_{A,p}+[\sigma,u]_{B,p'}\big)\|f\|_{L^p(\s)}.\]
\end{theorem}

\medskip

The remainder of this paper is organized as follows.  In
Section~\ref{sec:dyadic} we introduce the powerful notion of dyadic
grids on spaces of homogeneous type.  These were first constructed by
Christ~\cite{MR1104656}, but we will follow the more recent work of
Hyt\"onen and Kairema~\cite{MR2901199}.  These grids let us naturally
extend the Calder\'on-Zygmund decomposition and the techniques of the
so-called sparse operators to an SHT.  In
Section~\ref{sec:step1-step2} we will reduce the proof of our main
theorems to proving estimates for sparse operators.  The proof depends
on results that in the Euclidean case are due to
Lerner~\cite{Lern2012} and Lacey, Sawyer and
Uriarte-Tuero~\cite{LacSawUT2010}.  We give the corresponding results
for an SHT.  In Section~\ref{sec:double} we prove
Theorem~\ref{thm:double} by proving the corresponding result for
sparse operators.  The proof is nearly identical to the proof given
in~\cite{dcu-martell-perez} in the Euclidean case, so we only sketch
the details.  In Section~\ref{sec:step4} we prove a weak $(1,1)$
inequality for sparse operators that we need for our proof of
Theorem~\ref{thm:main}.  Our proof follows the broad outline of the
analogous result for singular integrals in Euclidean spaces due to
P\'erez~\cite{perez94b}; however, it is simpler because of the
localized behavior of sparse operators. In Section~\ref{sec:step5} we
prove Theorems~\ref{thm:main} and~\ref{thm:main-strong}.  Finally, in
Section~\ref{section:example} we construct a pair of weights on the
real line that satisfies a separated logarithmic bump condition but
not the corresponding double bump condition.

\bigskip

In our proofs of Theorems~\ref{thm:main} and~\ref{thm:main-strong}
the only place we use that $A$ and $B$ are log bumps is in the final
argument in Section~\ref{sec:step5}.  However, despite repeated
efforts we are unable to eliminate this assumption.  Nevertheless, we conjecture
that both results are true with the weaker assumption that $\bar{A}\in
B_{p'}$ and $\bar{B}\in B_p$, but we believe that new techniques will
be required to prove this.  On the other hand, very recently Nazarov, Reznikov
and Volberg~\cite{nazarov-reznikov-volbergP} have given a proof of the
separated bump result in Euclidean spaces using Bellman functions.
Certain aspects of their proof lead them to suggest that the full
conjecture may be false.

\section{Dyadic cubes in spaces of homogeneous type}
\label{sec:dyadic}

An important tool for our proofs is the concept of a dyadic grid $\D$
on an SHT and the concept a sparse family $S$ in $\D$.  These
generalize the classical Calder\'on-Zygmund decomposition
(cf.~\cite[Appendix~A]{MR2797562}).  The following result is due to
Hyt\"onen and Kairema~\cite{MR2901199} (see also
Christ~\cite{MR1104656}).

\begin{theorem}
\label{dyadic}
Given an SHT $(X,\rho,\mu)$, there exist constants $C>0$, $0<\eta,\,\epsilon<1$, depending on $X$, a family of sets $\D =\cup_{k\in
  \mathbb{Z}}\D_k$ (called a dyadic decomposition of $X$) and a
corresponding family of points $\{ x_c(Q)\}_{Q\in \D}$ that satisfy
the following properties:
\begin{enumerate}
\setlength{\itemsep}{6pt}
\item  for all $k\in \mathbb{Z}$, $\displaystyle X = \bigcup_{Q\in \D_k}Q$;

\item If $Q_1,Q_2\in \D\text{ then either }Q_1\cap Q_2 = \emptyset$,
  $Q_1 \subset Q_2$ or $Q_2\subset Q_1$;

\item For any $Q_1\in \D_k$ there exists at least one $Q_2\in
  \D_{k+1}$ (called a child of $Q_1$) such that $Q_2\subset Q_1$, and
  there exists exactly one $Q_3\in \D_{k-1}$  (called the parent of
  $Q_1$) such that $Q_1\subset Q_3$;

\item If  $Q_2$ is a child of $Q_1$, then  $\mu(Q_2)\geq \epsilon \mu(Q_1)$;

\item $B(x_c(Q), \eta^k)\subset Q \subset B(x_c(Q), C\eta^k)$.
\end{enumerate}
\end{theorem}

The sets $Q \in \D$ are referred to as dyadic cubes with center
$x_c(Q)$ and sidelength $\eta^k$,  but we must emphasize that
these are not cubes in any standard sense even if the underlying space
is $\R^n$, and care must be taken when visualizing them.  An exact
characterization of the kinds of sets which can be dyadic cubes is
given in~\cite{hytonen-kairemaP}.
Below we will need the dilations $\lambda Q$, $\lambda>1$, of dyadic
cubes.  However, these will actually be balls containing $Q$: given a cube $Q$, we define
\begin{equation*}
\lambda Q=B(x_c(Q), \lambda C \eta^k).
\end{equation*}

Families of dyadic grids can be constructed that have additional
useful properties:  see~\cite{MR2901199}.  We apply one such family to
show that our bump conditions can be restated in terms of dyadic
cubes.  Given a dyadic grid $\D$, a pair of weights $(u,\sigma)$, and
a Young function $A$, define
\[  [u,\sigma]_{A,p}^\D =\sup_{Q\in \D} 
\|u^{1/p}\|_{A,Q} \|\sigma^{1/p'}\|_{p',Q}. \]
We define $ [u,\sigma]_{A,B,p}^\D$ similarly.

\begin{lemma} \label{lemma:equiv-bump}
Given a pair of weights $(u,\sigma)$, and
Young functions $A$ and $B$, 
\[ [u,\sigma]_{A,p} \approx \sup_\D [u,\sigma]_{A,p}^\D, \qquad 
[u,\sigma]_{A,B,p} \approx \sup_\D [u,\sigma]_{A,B,p}^\D.\]
In both cases, the constants in the equivalence depend only on $X$.
\end{lemma}

\begin{proof}
We prove the first equivalence; the proof of the second is identical.
Given a dyadic grid $\D$ and $Q\in \D$, by Theorem~\ref{dyadic} there
exists a ball $B_\rho$ such that $Q\subset B_\rho$ and $\mu(B_\rho)\approx
\mu(Q)$.  Therefore, there exists $C(X)>1$ such that for any $\lambda>0$,
\[ \dashint_Q A\left(\frac{u(x)}{\lambda}\right)\,d\mu(x) \leq
C(X) \dashint_{B\rho} A\left(\frac{u(x)}{\lambda}\right)\,d\mu(x) \leq
\dashint_{B\rho} A\left(\frac{C(X)u(x)}{\lambda}\right)\,d\mu(x); \]
the last inequality holds since Young functions are convex.
Hence, by the definition of the Orlicz norm, $\|u\|_{A,Q} \leq
C(X)\|u\|_{A,B_\rho}$.  The same estimate holds for the norm of
$\sigma$. We thus have that 
\[ \sup_\D [u,\sigma]_{A,p}^\D \leq C(X) [u,\sigma]_{A,p}. \]
To prove the reverse inequality, we use the fact that there exists a
family of dyadic grids $\D^1,\ldots, \D^J$, $J$ depending only on $X$,
that satisfy the properties of Theorem~\ref{dyadic} with the
additional property that given any ball $B_\rho$, there exists $j$ and
$Q\in \D^j$ such that $B_\rho\subset Q$ and $\mu(B_\rho) \approx
\mu(Q)$.  (See~\cite[Theorem~4.1]{MR2901199}.)  Therefore, we can
repeat the above argument, reversing the roles of $B_\rho$ and $Q$, to get 
\[ [u,\sigma]_{A,p}  \leq C(X)\sup_\D [u,\sigma]_{A,p}^\D. \]
\end{proof}

\bigskip

Given a collection of dyadic cubes $\D$, a 
sparse family $S \subset \D$ is a collection of dyadic
cubes  for which there exists a collection of sets $\{E(Q) : Q \in S
\}$ such that the sets $E(Q)$ are pairwise disjoint, $E(Q)\subset Q$,
and $\mu(Q) \leq 2\mu(E(Q))$.   
Sparse families of cubes are a generalization of the
Calder\'on-Zygmund decomposition in the Euclidean case.  Using
Theorem~\ref{dyadic} we can form this decomposition in an SHT.    In
order to do this we need the Lebesgue differentiation theorem, which
holds in any SHT.  This fact seems to be new, though
the proof only consists of assembling pieces already present in the
literature: in particular, it is implicit in
Toledano~\cite{MR2061315}.

\begin{lemma} \label{lemma:LDT}
Given an SHT $(X,\rho,\mu)$, the Lebesgue differentiation theorem
holds:  for $\mu$-almost every $x\in X$, 
\begin{equation} \label{eqn:LDT1}
 \lim_{r\rightarrow 0}\frac{1}{\mu(B_\rho(x,r))}\int_{B_\rho(x,r)}
|f(y)-f(x)|\,d\mu(y) = 0.
\end{equation}
\end{lemma}

\begin{proof}
Mac\'\i as and Segovia, building on their earlier work in~\cite{MR546295},
showed in~\cite{macias-segovia80} that given any SHT $(X,\rho,\mu)$,
there exists an equivalent quasidistance $\delta$ (i.e., there exist
constants $c_1,\,c_2$ depending on $X$ such that for all $x,\,y\in X$,
$c_1\rho(x,y) \leq \delta(x,y) \leq c_2\rho(x,y)$), such that given
any ball $B_\delta$ with respect to $\delta$, then
$(B_\delta,\delta,\mu)$ is again a space of homogeneous type, and the
constants are independent of the ball $B_\delta$.

Toledano~\cite{MR2061315} proved that since $\mu(B_\delta)<\infty$,
the measure $\mu$ when restricted to $B_\delta$ is regular.  The
Lebesgue differentiation theorem holds for regular measures: this
follows from the standard argument
(cf. Rudin~\cite[Chapter~7]{MR924157}) using the fact that the maximal
operator is weak $(1,1)$ on $L^1(B_\delta,\mu)$
(Christ~\cite{MR1104656}) and that smooth functions of compact support
are dense in $L^1(B_\delta,\mu)$ (\cite[Chapter~3]{MR924157}).
Therefore, we have that for
$\mu$-almost every $x\in B_\delta$,
\[ \lim_{r\rightarrow 0}\dashint_{B_\delta(x,r)}
|f(y)-f(x)|\,d\mu(y) = 0. \]
Since $\rho$ and $\delta$ are equivalent and $\rho$ is doubling, it
follows immediately that \eqref{eqn:LDT1} holds in $B_\rho$.  Since $X$ can be
covered by a countable collection of $\rho$-balls, it holds for
$\mu$-almost every $x\in X$.
\end{proof}

\begin{remark}
  As a corollary to this proof we also have that $C_c^\infty(X)$ is
  dense in $L^1(X,\mu)$.  This fact, together with
  Lemma~\ref{lemma:LDT}, can be used to simplify the 
  hypotheses for results in a number of papers: see, for
  example,~\cite{MR2661130,MR0442579}.
\end{remark}

\begin{corollary} \label{cor:LDT-dyadic}
Given an SHT $(X,\rho,\mu)$ and a dyadic grid $\D$ that satisfies the
hypotheses of Theorem~\ref{dyadic}, then for $\mu$-almost every $x\in
X$, if $\{Q_k\}$ is the sequence of dyadic cubes in $\D$ such that
$\cap_k Q_k = \{x\}$, then
\begin{equation} \label{eqn:LDT-dyadic1}
 \lim_{k\rightarrow \infty}\dashint_{Q_k}
|f(y)-f(x)|\,d\mu(y) = 0.
\end{equation}
\end{corollary}

\begin{proof}
First note that since $\rho$ is a quasi-distance and $\mu$ is
doubling, if $x \in B(x_0,r)$, then $B(x_0,r)\subset B(x,2Kr)$ and 
$\mu(B(x,2Kr))\approx \mu(B(x_0,r))$.  Hence,
\[ \dashint_{B(x_0,r)} |f(y)-f(x)|\,d\mu(y) \leq
C\dashint_{B(x,2Kr)}|f(y)-f(x)|\,d\mu(y). \]
Therefore, if $B_k$ is a sequence of balls such that $\bigcap_k B_k
=\{x\}$, then it follows from Lemma~\ref{lemma:LDT} that
\begin{equation} \label{eqn:LDT-dyadic2}
 \lim_{k\rightarrow 0} \dashint_{B_k} |f(y)-f(x)|\,d\mu(y) = 0.
\end{equation}

Now for any $k$, by Theorem~\ref{dyadic} there exists a ball $B_k$
such that $x\in Q_k\subset B_k$ and $\mu(B_k) \leq C\mu(Q_k)$.
Then~\eqref{eqn:LDT-dyadic1} follows at once from~\eqref{eqn:LDT-dyadic2}.
\end{proof}

\begin{remark}
Corollary~\ref{cor:LDT-dyadic} was stated in~\cite{AV-2012} without
proof and with a reference to~\cite{MR2061315}.  However, as we noted,
this result was only implicit there.
\end{remark}

We now extend the Calder\'on-Zygmund decomposition to an SHT.  We give
a version that holds for Orlicz norms and not just for $L^1$ averages.
We begin by defining  a dyadic Orlicz maximal operator.  Given a dyadic grid $\D$
and a Young function $\Phi$, define
\[  M_\Phi^\D f(x) = \sup_{x\in Q\in \D} \|f\|_{\Phi,Q}. \]
The standard dyadic maximal operator is gotten by taking $\Phi(t)=t$;
in this case we simply write $M^\D$.

\begin{theorem} \label{thm:CZ-cubes} Given an SHT $(X,\rho,\mu)$  such
  that $\mu(X)=\infty$, a
  dyadic grid $\D$, and a Young function $\Phi$, suppose that $f$ is a
  measurable function such that $\|f\|_{\Phi,Q}\rightarrow 0$ as
  $\mu(Q)\to \infty$.  Then the following are true:

\begin{enumerate}

\item For each $\lambda>0$, there exists a collection 
$\{Q_j\}\subset \D$ that is pairwise disjoint, maximal with respect to
inclusion, and such that
\[
\Omega_\lambda = \{x\in X : M_\Phi^\D f(x)>\lambda\}= \bigcup_j Q_j.
\]
Moreover, there exists a constant $C(X)$ such that for every $j$,
\[
\lambda<\|f\|_{\Phi,Q_j}\le C(X)\,\lambda.
\]

\item Given $a>2/\epsilon$, where $\epsilon$ is as in
  Theorem~\ref{dyadic}, for each $k\in \Z$ let $\{Q_j^k\}_j$ be the
   collection of maximal dyadic cubes in $(1)$ with
$$
\Omega_k = \{x\in X: M_\Phi^\D f(x)>a^k\}= \bigcup_j Q_j^k.
$$
Then the set of cubes $S=\{Q_j^k\}$ is sparse, and $E(Q_j^k)=
Q_j^k\setminus \Omega_{k+1}$.
\end{enumerate}

If $\mu(X)<\infty$, then $(1)$ holds provided that $\lambda >
\dashint_X |f(x)|\,d\mu(x)$, and $(2)$ holds for all $k$ such that $a^k>\dashint_X |f(x)|\,d\mu(x)$.
\end{theorem}

The proof of Theorem~\ref{thm:CZ-cubes} is essentially identical to
that in the Euclidean in case: see, for
example,~\cite[Appendix~A.1]{MR2797562}.  The constant $C(X)$ in $(1)$
depends on the doubling constant of $\mu$.  When $\mu(X)<\infty$ some minor modifications to the proof are
necessary; these correspond to what is often referred to as a
``local'' Calder\'on-Zygmund decomposition.  To make them it suffices
to note that in this case $X$ is bounded (see~\cite{MR1791462}).
Therefore, by Theorem~\ref{dyadic}, for all dyadic cubes $Q$
sufficiently large, $X=Q$, and so the argument for $(1)$ still holds
if we take $\lambda>\dashint_X |f(x)|\,d\mu(x) = \dashint_Q
|f(x)|\,d\mu(x)$.

\medskip

\begin{theorem} \label{thm:CZ-decomp} Given an SHT $(X,\rho,\mu)$ such
  that $\mu(X)=\infty$, and a dyadic grid $\D$, suppose $f$ is a
  function such that $\|f\|_{1,Q}\rightarrow 0$ as $\mu(Q)\rightarrow
  \infty$.  Then for any $\lambda>0$ there exists a family
  $\{Q_j\}\subset \D$ and functions $b$ and $g$ such that:
\begin{enumerate}
\item $f=b+g$;

\item $g= f\chi_{\{ x : M^\D f(x)\leq \lambda \}}+ \sum_j f_{Q_j}$;

\item for $\mu$-a.e. $x\in X$, $|g(x)| \leq C(X)\lambda$;

\item $b= \sum_j b_j$, where $b_j = (f-f_{Q_j})\chi_{Q_j}$;

\item $\supp(b_j)\subset Q_j$ and $\dashint_{Q_j}
  b_j(x)\,d\mu(x)=0$. 
\end{enumerate}

If $\mu(X) < \infty$, then this decomposition still exists if we take
$\lambda>\dashint_X |f(x)|\,d\mu(x)$. 
\end{theorem}

Theorem~\ref{thm:CZ-decomp} is proved exactly as in the Euclidean
case, taking $\{Q_j\}$ to be the cubes from $(1)$ in
Theorem~\ref{thm:CZ-cubes}.  The proof that $g\in L^\infty$ requires
the Lebesgue differentiation theorem, Corollary~\ref{cor:LDT-dyadic}.

\section{Reduction to estimates for sparse operators}
\label{sec:step1-step2}

Given a dyadic grid $\D$ and sparse family $S$ in $\D$,  define the sparse
operator $T^S=T^{S,\D}$ by
 \[T^Sf(x) = \sum_{Q\in S}\left(\dashint_Q f \,d\mu\right)\cdot \chi_Q(x).\]
The operator $T^S$ is a positive, dyadic
Calder\'on-Zygmund operator.   It follows from the definition of
sparseness that $T^S$ is bounded on $L^2(\mu)$ and satisfies a weak
$(1,1)$ inequality:  see~\cite[Lemmas~6.4,~6.5]{AV-2012}.

\smallskip

A key feature of our proofs is that we reduce the problem for
Calder\'on-Zygmund operators to proving the same
estimates for sparse operators.  To do so we need to extend two results from the
Euclidean setting to spaces of homogeneous type.  The first result is due to
Lerner~\cite{Lern2012} in the Euclidean setting; it was central to his
greatly simplified proof of the $A_2$ conjecture.  We defer the proof
until the end of this section.

\begin{theorem} \label{thm:lerner}
Given an SHT $(X,\rho,\mu)$ and a Calder\'on-Zygmund operator $T$,
then for any Banach function space $Y$, 
\begin{equation} \label{eqn:lerner1}
\|T(f\s)\|_Y\leq C(X,T)\sup_{\D,S} \|T^S(f\s)\|_Y,
\end{equation}
where the supremum is taken over every dyadic grid $\D$ in $(X,\rho,\mu)$
and over every sparse family $S$ in $\D$. 
\end{theorem}

\smallskip

By taking
$Y$ equal to $L^{p,\infty}(u)$ or $L^p(u)$, 
it follows immediately from Theorem~\ref{thm:lerner} that to prove
estimates for Calder\'on-Zygmund operators, it suffices to prove them for all sparse operators $T^S$ with
constants independent of $S$ and $\D$.     Below, we will prove
Theorems~\ref{thm:double}, \ref{thm:main}  by establishing such estimates.

To prove Theorem~\ref{thm:main-strong} we need to argue indirectly
using a result which connects the weak and strong type norm
inequalities of sparse operators.   In the Euclidean case this
theorem is due to Lacey, Sawyer and Uriarte-Tuero~\cite{LacSawUT2010}.

 \begin{theorem} \label{thm:LSUT}
Given an SHT $(X,\rho,\mu)$, let $\D$ be a dyadic grid and $S$ a
sparse family in $\D$.  Then
\begin{equation} \label{eqn:lacey-strong}
\|T^S(\cdot\s)\|_{L^p(\s)\to L^p(u)}\approx 
\|T^S(\cdot\s)\|_{L^p(\s)\to L^{p,\infty}(u)} +
\|T^S(\cdot u)\|_{L^{p'}(u)\to L^{p',\infty}(\s)}
\end{equation}
The constants in the equivalence depend only on $X$,
$T$ and $p$; in particular they are independent of $\D$ and $S$.  
\end{theorem}

The proof of Theorem~\ref{thm:LSUT} passes through the equivalence of
the weak and strong type inequalities to certain testing conditions.
The proof of this equivalence for weak type inequalities in an SHT is the same as in
\cite[Section~2.2]{LacSawUT2010} in the Euclidean setting; it is
straightforward to see that the only properties of dyadic cubes used
in the proof are the those given in Theorem~\ref{dyadic}.   
The proof of this equivalence  for strong type inequalities in~\cite{LacSawUT2010} is much
more involved; however, a simpler proof was recently given by
Treil~\cite{treil-2012} and as he notes (see Section~5 of his paper),
this proof also extends to an SHT with essentially no change.

\smallskip

Given Theorem~\ref{thm:LSUT},  Theorem~\ref{thm:main-strong}
follows from the characterization of the weak type inequality in
Theorem~\ref{thm:main}.   We will make this precise in
Section~\ref{sec:step5} below.
\medskip

\subsection*{Proof of Theorem~\ref{thm:lerner}}
Our proof draws heavily on the results proved in~\cite{AV-2012} and we
refer the reader there for complete details.  

By our assumptions on $f$ and $\sigma$ we can, for clarity, replace $f\sigma$ by
$f$.  As was proved
in~\cite[Proposition~4.3]{MR2901199}, if we fix a point $x_0\in X$, we
can construct a dyadic grid $\D^*$ satisfying Theorem~\ref{dyadic} that
contains a sequence of nested dyadic cubes $\{Q_N\}$ such that $x_0$
is the center of each cube $Q_N$ and such that $\bigcup_N Q_N = X$.
Therefore, by duality and Fatou's lemma, there exists $g$ in the
associate space $Y'$, $\|g\|_{Y'}=1$, such that
\[ \|Tf\|_Y = \int_X |Tf(x)|g(x)\,d\mu(x)
\leq \liminf_{N\rightarrow \infty} \int_{Q_N} |Tf(x)|g(x)\,d\mu(x). \]
Fix $N>0$; we will prove that the final integral is bounded by $C\sup
\|T^S f\|_Y$, where the supremum is taken over some collection of
$S$ and $\D$, but the constant is independent of these and also
independent of $N$.  

As was proved in~\cite[Section~5]{AV-2012}, there exist
$C_1,\,C_2,\,\eta>0$ such that for $\mu$-a.e.
$x\in Q_N$,  
\begin{equation} \label{repT}
|Tf(x)|\leq
C_1 Mf(x)+ C_2\sum_{i=1}^{\infty}
\frac{1}{2^{i\eta}} A_if(x),
\end{equation}
where \[A_if(x) = \sum_{Q\in S_N} \dashint_{2^i
  Q}f(y)\,d\mu(y)\cdot\chi_{Q}(x)
\]
and $S_N$ is a sparse subset of $\D^*$ that consists of dyadic
sub-cubes of $Q_N$.  The constants depend only on $X$ and $T$; in
particular $C_1$ depends on the fact (see \cite{MR1104656}) that $T :
L^1(X,\mu)\rightarrow L^{1,\infty}(X,\mu)$.  Therefore, by H\"older's
inequality (with respect to $Y$ and $Y'$),
\[ \int_{Q_N} |Tf(x)|g(x)\,d\mu(x) \leq C_1\|Mf\cdot \chi_{Q_N}\|_Y +
C_2\sum_i 2^{-i\eta} \|A_i f\cdot \chi_{Q_N}\|_Y=I_1+I_2. \]

To estimate $I_1$ 
we give a pointwise estimate for $Mf(x)$.
By~\cite[Theorem~7.9]{MR2901199} there exists a constant $K=K(X)$ and
a collection
$\D^1,\ldots,\D^K$ of dyadic grids such that for every $x\in X$,
\begin{equation} \label{eqn:lerner2}
Mf(x) \leq C(X)\sum_{k=1}^K M^kf(x),
\end{equation}
where $M^k=M^{\D^k}$ is the dyadic maximal operator defined with
respect to $\D^k$.
We claim that for each $k$ there exists a sparse subset $S_k$
(depending on $f$) such that 
\begin{equation} \label{eqn:lerner3}
M^kf(x) \leq C(X)T^{S_k}f(x).
\end{equation}
This follows from $(2)$ in Theorem~\ref{thm:CZ-cubes}.
With the notation of this result, let $S_k=\{Q_j^i\}\in\D^k$ be the
sparse family.  Then given $x\in \Omega_i\setminus \Omega_{i+1}$,
there exists $Q^i_j$ such that
\[ Mf(x) \leq a^{i+1} < a\dashint_{Q_j^i} f(y)\,d\mu(y); \]
hence, for $\mu$-a.e. $x$,
\[ M^kf(x) \leq a\sum_{i,j} \dashint_{Q_j^i} f(y)\,d\mu(y) \cdot
\chi_{Q_j^i}(x) = aT^{S_k}f(x). \]

If we now combine inequalities~\eqref{eqn:lerner2}
and~\eqref{eqn:lerner3}, we have that
\[ I_1 \leq C(T,X)\|Mf\|_Y \leq  C(T,X)\sum_{k=1}^K \|T^{S_k}f\|_Y \leq C(T,X)
\sup_{\D,S}\|T^S f\|_Y. \]

\medskip

To estimate $I_2$ we will decompose each $A_i f$ and
apply duality.    By~\cite[Theorem~4.1]{MR2901199} there exists a
family of dyadic grids $\D^1,\ldots,\D^J$, satisfying the properties
of Theorem~\ref{dyadic} with the additional property that given any
ball $B_\rho$, there exists $j$ and $Q^*\subset \D^j$ such that
$B_\rho\subset Q^*$ and $\mu(B_\rho) \approx
\mu(Q^*)$, with constants depending only on $X$.  Recall (see the
discussion after Theorem~\ref{dyadic}) that $2^iQ$ is defined to be a
ball.  Therefore, if we define
\[ S_N^j = \{ Q\in S_N : \exists Q^* \in \D^j, 2^iQ \subset Q^*\}, \]
then
\[ A_i f(x) \leq C(X) \sum_{j=1}^J \sum_{\substack{Q\in S_N^j\\2^i Q
    \subset Q^*}} \dashint_{Q^*} f(y)\,d\mu(y)\cdot \chi_Q(x)
= C(X) \sum_{j=1}^J B_{i,j}f(x). \]

Arguing as in~\cite[Section~6]{AV-2012} (see especially Lemmas~6.5
and~6.13) we apply the same argument used to prove~\eqref{repT} for
$T$ to the adjoint operators $B_{i,j}^*$.   Key to this is the fact
that adjoint operators are weak $(1,1)$ with a constant that is linear
in $i$.  This yields the following pointwise estimate:
\begin{multline*}
 B_{i,j}^*f(x) \leq iC_1(X) Mf(x) + i C_2(X) \sum_{Q\in S^j_*}
\dashint_Q f(y)\,d\mu(y)\cdot \chi_Q(x)  
\\ = iC_1(X) Mf(x) + i C_2(X)
T_{i,j}f(x),
\end{multline*}
where $S^j_*\subset \D^j$ is sparse.  Therefore, repeating
the above argument for bounding the maximal operator, we have that
$B_{i,j}^*$ is bounded pointwise by a finite sum of sparse operators $T^{S^l}$,
$1\leq l \leq L$
(defined with respect to a finite collection of dyadic grids $\D$).
We can now estimate $I_2$ by duality using the fact that the operators
$T^{S^l}$ are self-adjoint:  there exists a collection of
$g_i \in Y'$, $\|g_i\|_{Y'}=1$, such that 
\begin{align*}
I_2 
& = C(T,X) \sum_i 2^{-i\eta} \|A_i f \cdot \chi_{Q_N}\|_Y  \\
& = C(T,X)  \sum_i 2^{-i\eta} \int_{Q_N} A_if(x) g_i(x)\,d\mu(x) \\
& \leq C(T,X)  \sum_i 2^{-i\eta} \sum_{j=1}^J \int_{Q_N} B_{i,j}f(x) g_i(x)\,d\mu(x) \\
& \leq C(T,X)  \sum_i 2^{-i\eta} \sum_{j=1}^J \int_X f(x) B_{i,j}^*g_i(x)\,d\mu(x)
\\
& \leq C(T,X)  \sum_i i2^{-i\eta} \sum_{j=1}^J \sum_{l=1}^L \int_X
f(x) T^{S^l}g_i(x)\,d\mu(x) \\
& =  C(T,X)  \sum_i i2^{-i\eta} \sum_{j=1}^J \sum_{l=1}^L \int_X
T^{S^l}f(x) g_i(x)\,d\mu(x) \\
& \leq C(T,X) \sum_i i2^{-i\eta} \sum_{j=1}^J \sum_{l=1}^L
\|T^{S^l}f\|_Y \|g_i\|_{Y'} \\
& \leq C(T,X) \sup_{\D,S} \|T^S f\|_Y.
\end{align*}
This completes the proof of Theorem~\ref{thm:lerner}.

\section{Proof of Theorem~\ref{thm:double}}
\label{sec:double}

We will prove this result for sparse operators, with the
$[u,\sigma]_{A,B,p}$ condition replaced by the $[u,\sigma]_{A,B,p}^\D$
condition.  Theorem~\ref{thm:double} then follows immediately by
Theorem~\ref{thm:LSUT} and Lemma~\ref{lemma:equiv-bump}.  The proof
for sparse operators is essentially identical to the proof in the Euclidean
case in~\cite{dcu-martell-perez}; for the convenience of the reader we
sketch the details.

We need one preliminary result.  In the Euclidean case this is due to 
P\'erez~\cite{perez95}, and in an SHT to P\'erez and
Wheeden~\cite{MR1818113} and Pradolini and
Salinas~\cite{MR2022366}.   In the latter papers the proofs are for
maximal operators defined with respect to balls instead of dyadic
cubes, but the proofs rely on a version of Theorem~\ref{thm:CZ-cubes}
for balls and so immediately adapt to this setting.  

\begin{lemma} \label{lemma:orlicz-max}
Given an SHT $(X,\rho,\mu)$ and  a Young function $\Phi$ such that $\Phi\in B_p$, 
then 
\[ \|M_\Phi^\D f\|_{L^p(\mu)} \leq 
C(X) [\Phi]_{B_p}^{1/p}\| f\|_{L^p(\mu)}. \]
\end{lemma}

\begin{remark}
In \cite{perez95,MR2022366} it is assumed that $\Phi$ satisfies the
doubling condition $\Phi(2t)\leq C\Phi(t)$.  However, as noted
in~\cite[p.~102]{MR2797562} this assumption is only used to prove an
equivalent formulation of the $B_p$ condition.
\end{remark}

\begin{proof}[Proof of Theorem~\ref{thm:double}]  By duality and the
definition of $T^S$, there exists $g\in L^{p'}(u)$, $\|g\|_{L^{p'}(u)}=1$, such that
\begin{align*}
\|T^S(f\sigma)\|_{L^p(u)} 
& = \int_X T^S(f\sigma)(x)u(x)g(x)\,d\mu(x) \\
& \leq  2\sum_{Q\in S} \dashint_Q f(x)\sigma(x)\,d\mu(x) \dashint_Q
u(x)g(x)\,d\mu(x) \cdot \mu(E(Q)) \\
& \leq 8 \sum_{Q\in S} \|f\sigma^{1/p}\|_{\bar{B},Q} \|gu^{1/p'}\|_{\bar{A},Q} 
\|u^{1/p}\|_{A,Q}\|\sigma^{1/p'}\|_{B,Q} \mu(E(Q)) \\
& \leq 8[u,\sigma]_{A,B,p} ^\D\int_X M_{\bar{B}}^\D (f\sigma^{1/p})(x)
M_{\bar{A}}^\D (gu^{1/p'})(x)\,d\mu(x) \\
& \leq  8[u,\sigma]_{A,B,p}^\D \|M_{\bar{B}}^\D (f\sigma^{1/p})\|_p 
\|M_{\bar{A}}^\D (gu^{1/p'})\|_{p'} \\
& \leq C(X)[u,\sigma]_{A,B,p}^\D[\bar{A}]_{B_{p'}}^{1/p'}
[\bar{B}]_{B_p}^{1/p}\|f\|_{L^p(\sigma)} \|g\|_{L^{p'}(u)}.
\end{align*}
\end{proof}

In the next section we will need an equivalent version of this result
for sparse operators,
and so we state it here.  The equivalence is easily seen by letting
$\sigma=v^{1-p'}$. 

\begin{theorem} \label{thm:double-alt}
Given an SHT $(X,\rho,\mu)$, a dyadic grid $\D$ and a sparse family
$S\subset \D$, and Young functions $A,\,B$ with
$\bar{A}\in B_{p'}$ and $\bar{B}\in B_p$, suppose the pair of weights $(u,v)$
satisfies
\begin{equation} \label{eqn:doublt-alt1}
[[u,v]]_{A,B,p}^\D = \sup_{Q\in \D}\|u^{1/p}\|_{A,Q}\|v^{-1/p}\|_{B,Q} < \infty.
\end{equation}
Then 
\[ \|T^Sf\|_{L^p(u)} \leq C(X) [[u,v]]_{A,B,p}^\D\; [\bar{A}]_{B_{p'}}^{1/p'}\;
  [\bar{B}]_{B_{p}}^{1/p}\;
\|f\|_{L^p(v)}. \]
\end{theorem}

\section{A weak $(1,1)$ inequality}
\label{sec:step4}

In this section we prove a two-weight, weak $(1,1)$ inequality for
sparse operators.  A version of this result for general
Calder\'on-Zygmund operators in the Euclidean case was due to
P\'erez~\cite{perez94b} and our proof closely follows his.  However,
it is simplified because we are working with sparse operators:
instead of appealing to duality and the Coifman-Fefferman inequality
relating singular integrals and the maximal operator, we use
two-weight theory via Theorem~\ref{thm:double-alt}.

\begin{theorem} \label{thm:perez} Given an SHT $(X,\rho,\mu)$, let
  $\D$ be a dyadic grid satisfying the hypotheses of
  Theorem~\ref{dyadic}, and let $S\subset \D$ be sparse.  Let $\Phi$
  be a Young function such that for some $1<q<\infty$,
  $A_\Phi(t)=\Phi(t^q)$ satisfies $\bar{A}_\Phi\in B_{q'}$.  Then for
  all $\lambda>0$,
\begin{equation}
\label{Perezweak}
u(\{x \in X : T^Sf(x)>\lambda\})\leq 
C(X,\Phi,q)\frac{1}{\lambda}\int_X f(x)M_{\Phi}^{\D}u(x)d\mu(x).
\end{equation}
\end{theorem}

\begin{proof}
We first consider the case when $\mu(X)=\infty$; at the end of the
proof we will sketch the changes needed when $\mu(X)<\infty$.

Fix $\lambda>0$ and let the disjoint cubes $\{Q_j\}$ and functions $g$ and
$b=\sum_j b_j$ be as given by Theorem~\ref{thm:CZ-decomp}.
Since $f=g+b$, we have that
\begin{align*}
& u(\{x\in X : T^Sf(x)>\lambda\}) \\ 
& \qquad =  u(\Omega)
 + u(\{x\in \Omega^c : |T^Sb(x)|>\lambda/2\}) +u(\{x\in \Omega^c :
 T^Sg(x)>\lambda/2\}) \\ 
& \qquad = I_1+I_2+I_3,
\end{align*}
where $\Omega=\bigcup_j Q_j$.

The estimate for $I_1$ is immediate:  since $\mu$ is doubling, by the
properties of the cubes $\{Q_j\}$,
\begin{multline*}
 I_1 = u(\Omega)\leq
 \sum_ju(Q_j)=\sum_j\frac{u(Q_j)}{\mu(Q_j)}\mu(Q_j) 
\leq \frac{1}{\lambda}\sum_j\frac{u(Q_j)}{\mu(Q_j)}\int_{Q_j}f(x)d\mu(x) \\
\leq \frac{1}{\lambda}\sum_j\int_{Q_j}f(x)M^\D u(x) d\mu(x)
 \leq C(\Phi)\frac{1}{\lambda}\int_X f(x) M_{\Phi}^\D u(x) d\mu(x);
\end{multline*}
the last inequality follows from the fact since $t \lesssim \Phi(t)$,
$\|u\|_{1,Q}\leq C(\Phi)\|u\|_{\Phi,Q}$.

\bigskip

To estimate $I_2$, fix $x\in \Omega^c$; then $x\in Q_j^c$
for all $j$.  By linearity,  $T^Sb(x) = \sum_jT^Sb_j(x)$, and for each
$j$, 
\begin{multline*}
T^Sb_j(x) = \sum_{Q\in S}\dashint_Qb_j(y)d\mu (y) \cdot \chi_Q(x) \\
= \sum_{Q\in
  S}\frac{\mu(Q_j)}{\mu(Q)}\frac{1}{\mu(Q_j)}\int_Q(f(y)-f_{Q_j})\chi_{Q_j}(y)d\mu
(y) \cdot \chi_Q(x).
\end{multline*}
For $T^Sb_j(x)\neq 0$, we need $x\in Q$, which in turn implies that
$Q\cap Q_j\neq \emptyset$ and $Q\cap Q_j^c\neq\emptyset$. 
Since $Q,\,Q_j\in \D$, we must have that $Q_j\subset Q$.
But then
 \[\frac{1}{\mu(Q_j)}\int_Q(f(y)-f_{Q_j})\chi_{Q_j}(y)d\mu (y) =
\frac{1}{\mu(Q_j)}\int_{Q_j}\left( f(y)-f_{Q_j}\right) d\mu (y) = 0; \] 
Hence,  $T^Sb_j(x) =0$ and so $I_2=0$.

\bigskip

To estimate $I_3$ we want to apply Theorem~\ref{thm:double-alt} with
the pair $(u,M_\Phi u)$.  Let
$B(t)=t^{(rq)'}$ with $1/q<r<1$; then  $\bar{B}\in B_q$ and
$[\bar{B}]_{B_q} \leq C(q)$.  We claim that
\[ [[u,M_\Phi^\D u]]_{A_\Phi,B,p}^\D \leq 1. \]
To see this, fix $Q\in \D$.  Since $B(1)=1$, it follows that
$\|\chi_Q\|_{B,Q}=1$.  Moreover, for any $x\in Q$, by a change of
variables in the definition of the Orlicz norm,
\[ M_\Phi^\D u(x) \geq \|u\|_{\Phi,Q}
=\|u^{1/q}\|_{A_\Phi,Q}^q. \]
Therefore, 
\[ \|u^{1/p}\|_{A_\Phi,Q}\|M_\Phi^\D (u)^{-1/p}\|_{B,Q} 
\leq 
\|u^{1/p}\|_{A_\Phi,Q}\|u^{1/p}\|_{A_\Phi,Q}^{-1}\|\chi_Q\|_{B,Q}=1. \]

Hence, by Theorem~\ref{thm:double-alt} and since $g(x)\leq C(X)\lambda$,
\begin{align*}
I_3 
& \leq \frac{2^q}{\lambda^q}\int_{{\Omega^c}} T^Sg(x)^q u(x)\,d\mu(x) \\
& \leq C(X,\Phi,q)\;\frac{1}{\lambda^q}\int_X
  g(x)^qM_{\Phi}^\D (u\chi_{{\Omega^c}})(x)d\mu(x) \\
& \leq C(X,\Phi,q)\frac{1}{\lambda}\int_X
g(x)M_{\Phi}^\D (u\chi_{{\Omega^c}})(x)d\mu(x) \\
& = C(X,\Phi,q)\frac{1}{\lambda}\int_{\Omega^c}
 f(x)M_{\Phi}^\D (u\chi_{{\Omega^c}})d\mu(x) \\
& \qquad \qquad  +
 C(X,\Phi,q)\frac{1}{\lambda}\sum_j\int_{Q_j}
 f_{Q_j}M_{\Phi}^\D (u\chi_{{\Omega^c}})(x)d\mu(x) \\
& = J_1+J_2.
\end{align*}

Clearly,
\[ J_1\leq C(X,\Phi,q)\frac{1}{\lambda}\int_X f(x)M_{\Phi}^\D u(x)d\mu(x) \]
as desired.  To estimate $J_2$, assume for the moment that for each $j$ and $x\in Q_j$,
\begin{equation} \label{eqn:perez-weak1}
M_{\Phi}^\D(u\chi_{\Omega^c})(x)  \leq \inf_{y\in
  Q_j}M_{\Phi}^\D (u\chi_{Q_j^c})(y). 
\end{equation}
Given this,  we have that 
\begin{align*}
J_2
& \leq C(X,\Phi,q)\frac{1}{\lambda}\sum_jf_{Q_j}\mu(Q_j)\inf_{y\in
  Q_J}M_{\Phi}^D(u\chi_{Q_j^c})(y) \\
& \leq C(X,\Phi,q)\frac{1}{\lambda}\sum_j
\int_{Q_j}f(y)M_{\Phi}^\D(u\chi_{Q_j^c})(y)d\mu(y) \\
& \leq C(X,\Phi,q)\frac{1}{\lambda}\int_Xf(y)M_{\Phi}^\D u(y)d\mu(y).
\end{align*}

It remains to prove~\eqref{eqn:perez-weak1}.  
But if $x\in Q_j$, then
\[ M_\Phi^\D(u\chi_{{\Omega^c}})(x)  \leq 
M_\Phi^\D (u\chi_{Q_j^c})(x) = 
\sup_{x\in Q\in\D}\| u \chi_{Q_j^c}\|_{\Phi,Q}. \] 
The norm on the right hand side is non-zero only if $x\in Q$ and
$Q\cap Q_j^c \neq \emptyset$.  Therefore, by the properties of
dyadic cubes we must have that $Q\subset Q_j$.    Hence,
\[ M_\Phi^\D (u\chi_{Q_j^c})(x) = 
\sup_{ \substack{Q\in\D\\ Q_j\subset Q}}\| u \chi_{Q_j^c}\|_{\Phi,Q}, \]
and since this quantity is independent of $x\in Q_j$, we get
\eqref{eqn:perez-weak1}.

\bigskip

If $\mu(X)<\infty$, then we can repeat the above proof for all
$\lambda >\dashint_X f(x)\,d\mu(x)$.  If the opposite inequality
holds, then for some dyadic cube $Q$ sufficiently large, $Q=X$, and so
\begin{multline*}
u(\{ x\in X : T^Sf(x) > \lambda \})
\leq u(X) \leq \frac{u(Q)}{\mu(Q)}\frac{1}{\lambda} \int_Q
f(x)\,d\mu(x) \\
\leq \frac{1}{\lambda} \int_Q
f(x)M^\D u(x)\,d\mu(x)
\leq C(\Phi)\frac{1}{\lambda} \int_Q
f(x)M^\D_\Phi u(x)\,d\mu(x).
\end{multline*}
\end{proof}

\section{Proof of Theorems~\ref{thm:main} and~\ref{thm:main-strong}}
\label{sec:step5}

We first show that Theorem~\ref{thm:main-strong} is a consequence of
Theorem~\ref{thm:main}.   Given both separated bump conditions, the
latter result implies that
\begin{gather*} 
\|T^S(\cdot\s)\|_{L^p(\s)\to L^{p,\infty}(u)} 
\lesssim [u,\sigma]_{A,p}, \\
\|T^S(\cdot u)\|_{L^{p'}(u)\to L^{p',\infty}(\s)}
\lesssim [\sigma,u]_{B,p'}.
\end{gather*}
Therefore, by Theorem~\ref{thm:LSUT} we get the desired strong type
inequality. 

\bigskip

To prove Theorem~\ref{thm:main} it will again suffice to prove it for
sparse operators.  In order to do this we need a weighted norm inequality for
an Orlicz maximal operator.  The following result was proved
in~\cite{MR1713140} for the Hardy-Littlewood maximal operator in the
Euclidean case; the proof in a SHT is nearly the same and we sketch
the details. 

\begin{lemma} \label{lemma:bump-max} Given
  $1<p<\infty$, let $A,C$ and $\Phi$ be Young functions such that
  $A^{-1}(t)C^{-1}(t)\leq c_0\Phi^{-1}(t)$ for $t>0$ where 
  and $C\in B_{p'}$.  If $[u,\s]_{A,p}^\D<\infty$, then 
\[ \|M_\Phi^\D( f u)\|_{L^{p'}(\sigma)}
\leq C(X,c_0) [u,\s]_{A,p}[C]_{B_{p'}}^{1/p'}    \|f\|_{L^{p'}(u)}. \]
\end{lemma}

\begin{proof}
We first consider the case when $\mu(X)=\infty$.  
By Theorem~\ref{thm:CZ-cubes}, fix $a>1$ sufficiently large and form
the cubes $\{Q_j^k\}$ such that
\[ \Omega_k = \{ x : M_\Phi^\D(fu)(x) > a^k \} = \bigcup_j Q_j^k. \]
Then by the generalized H\"older's inequality  we have that
\begin{align*}
\int_X M_\Phi^\D(fu)(x)^{p'} \sigma(x)\,d\mu(x)
& = \sum_k \int_{\Omega_k\setminus \Omega_{k+1}}
 M_\Phi^\D(fu)(x)^{p'} \sigma(x)\,d\mu(x) \\
& \leq a^{p'} \sum_{k,j} \|fu\|_{\Phi,Q_j^k}^{p'} \sigma(Q_j^k) \\
& \leq C(X,c_0) \sum_{k,j} \|fu^{1/p'}\|_{C,Q_j^k}^{p'}
\|u^{1/p}\|_{A,Q_j^k}^{p'} \|\sigma^{1/p'}\|_{p',Q_j^k}^{p'} E(Q_j^k) \\
& \leq C(X,c_0)[u,\sigma]_{A,p}^{p'}\sum_{k,j} 
\int_{E(Q_j^k)} M_C^\D (fu^{1/p'}(x)^{p'})\,d\mu(x) \\
& \leq C(X,c_0)[u,\sigma]_{A,p}^{p'}
\int_{X} M_C^\D (fu^{1/p'}(x)^{p'})\,d\mu(x) \\
& \leq C(X,c_0) [u,\sigma]_{A,p}^{p'} [C]_{B_{p'}}
\int_X f(x)^{p'}u(x)\,d\mu(x).
\end{align*}

\bigskip

If $\mu(X)<\infty$, then let $k_0$ be the largest integer such that
\[ a^{k_0} < \|f\|_{X,\Phi}. \]
Then $\Omega_{k_0}=X$.  We can repeat the above argument summing over
$k\geq k_0$, and for $k>k_0$ we can still form the cubes $\{Q_j^k\}$
and argue as before.  When $k=k_0$, then there exists a large dyadic cube
$Q=X$. Hence, $a^{k_0}< \|f\|_{\Phi,Q}$ and the argument proceeds as
before, replacing the collection $\{Q_j^k\}$ with the single cube $Q$.
\end{proof}

We can now prove Theorem~\ref{thm:main}.  Note that this is the only
part of the proof in which we use the fact that $A$ is a log bump.
The proof uses an extrapolation argument from~\cite{MR1713140}; see
also~\cite[Chapter 8]{MR2797562}.

Fix $\lambda>0$ and define 
\[\Omega_{\lambda} = \{x\in X : T^S(f\s)(x)>\lambda\}. \] 
Then by duality, there exists $h\in L^{p'}(u)$, $\|h\|_{L^{p'}(u)}=1$, such
that 
\[u(\Omega_{\lambda})^{1/p} = \|\chi_{\Omega_\lambda}\|_{L^p(u)}
= \int_{\Omega_{\lambda}}u(x)h(x)d\mu(x)  =uh(\Omega_\lambda).  \] 

Now let $\Phi(t)=t\log(e+t)^\epsilon$, where we will fix the value of $\epsilon>0$
below.  Let $q-1=\epsilon/2$ and let
$A_\Phi(t)=t^q\log(e+t)^{q-1+\epsilon/2}$.  Then $\bar{A}_\Phi \in B_{q'}$,
and so by Theorem~\ref{thm:perez},
\begin{multline*}
uh(\Omega_\lambda) \leq C(X,\epsilon)
\frac{1}{\lambda}\int_Xf(x)\s(x)\, M_{\Phi}^\D(uh)(x)d\mu(x)  \\
\leq C(X,\epsilon)\frac{1}{\lambda} \|f\|_{L^p(\s)}
\|M_\Phi^\D(uh)\|_{L^{p'}(\sigma)}.
\end{multline*}

Now fix $\epsilon<\delta/p$ and define
$C(t)=t^{p'}\log(e+t)^{-1-(p'-1)\eta}$, where $\eta=\delta-\epsilon
p$.  Then $C\in B_{p'}$ and $[C]_{B_{p'}}$ depends only on $p$ and
$\delta$.  Moreover, we have that 
$A^{-1}(t)C^{-1}(t) \leq c_0 \Phi^{-1}(t)$, where the constant $c_0$
depends on $\delta$ and $p$.  (See~\cite{MR1713140} for details.) 
Therefore, by
Lemma~\ref{lemma:bump-max},
\[ \|M_\Phi^\D(hu)\|_{L^{p'}(\sigma)} \leq
C(X,p,\delta)[u,\sigma]_{A,p}\|h\|_{L^{p'}(u)} =
C(X,p,\delta)[u,\sigma]_{A,p}. \]
Combining the above inequalities we get the desired result.  

\section{Separated and double bump conditions}
\label{section:example}

We construct our example on the real line with $p=2$.  Our example can
be readily modified to work for other values of $p$.  Define the Young functions
\[ A(t)=B(t)= t^2 \log(e+t)^2.  \]
Then $\bar{A},\,\bar{B} \in B_2$. By rescaling, if we let
$\Phi(t)=t\log(e+t)^2$, then for any pair $(u,\sigma)$, 
\[ \|u^{1/2}\|_{A,Q} \approx \|u\|_{\Phi,Q}^{1/2}, \qquad
\|\s^{1/2}\|_{B,Q} \approx \|\s\|_{\Phi,Q}^{1/2}. \]
Therefore, it will suffice  estimate the norms of $u$ and $\sigma$ with respect
to $\Phi$.   Similarly, we can replace the localized $L^2$ norms of
$u^{1/2}$ and $\sigma^{1/2}$ with the $L^1$ norms of $u$ and $\sigma$.

Before we define $u$ and $\sigma$ we first construct a pair
$(u_0,\sigma_0)$ which will be the basic building block for our
example.  Fix an integer $n\geq 2$ and  define
$Q=(0,n)$, $\sigma_0=\chi_{(0,1)}$ and   $u_0 = K_n\chi_{(n-1,n)}$,
where $K_n=n^2\log(e+n)^{-3}$.   Then a straightforward estimate with
the definition of the Orlicz norm snows that 
\[ \|u_0\|_{1,Q} = \frac{K_n}{n}, \;
\|u_0\|_{\Phi,Q} \approx \frac{K_n\log(e+n)^2}{n},  \qquad 
\|\sigma_0\|_{1,Q} = \frac{1}{n}, \; 
\|\sigma_0\|_{\Phi,Q} \approx\frac{\log(e+n)^2}{n}. \]
Therefore, we have that 
\[  \|u_0\|_{1,Q}\|\sigma_0\|_{\Phi,Q}, \quad
\|u_0\|_{\Phi,Q}\|\sigma_0\|_{1,Q} \approx \frac{1}{\log(e+n)}, \]
but
\[ \|u_0\|_{\Phi,Q}\|\sigma_0\|_{\Phi,Q} \approx \log(e+n). \]

\medskip

We now define $u$ and $\sigma$ as follows:
\[ 
u(x) = \sum_{n\geq 2} K_n\chi_{I_n}(x) \quad
\sigma(x) = \sum_{n\geq 2} \chi_{J_n}(x).
\]
where $I_n=(e^n+n-1,e^n+n)$ and $J_n= (e^n,e^n+1)$.
Since the above computations are translation invariant, we immediately
get that if $Q_n=(e^n,e^n+n)$, then 
\[ \|u\|_{\Phi,Q_n}\|\sigma\|_{\Phi,Q_n} \approx \log(e+n), \]
and so $[u,\sigma]_{A,B,2}=\infty$.   

It remains, therefore, to show
that $[u,\sigma]_{A,2}$ and $[\sigma,u]_{B,2}$ are both finite.  We
will consider $[u,\sigma]_{A,2}$; the argument for the second is
essentially the same.    Fix an interval $Q$; we will show that
$\|u\|_{\Phi,Q}\|\sigma\|_{1,Q}$  is uniformly bounded.   Fix an integer $N$ such that 
$N-1 \leq |Q| \leq N$.   We need to consider those values of $n$ such
that $Q$ intersects either $I_n$ or $J_n$.  

Suppose that for some $n\geq N+2$, $Q$ intersects $I_n$.  But in this
case it cannot intersect $J_k$ for any $k$ and so
$\|\sigma\|_{1,Q}=0$.  Similarly, if $Q$ intersects $J_n$, then
$\|u\|_{\Phi,Q}=0$. 

Now suppose that for some $n<N+2$, $Q$ intersects one of $I_n$ or
$J_n$.  If $\log(N) \lesssim n$ (more precisely, if
$N<e^n-e^{n-1}-1$), then for any $k\neq n$, $Q$ cannot intersect $I_k$
or $J_k$.  In this case 
$\|u\|_{\Phi,Q}\|\sigma\|_{1,Q}\neq 0$ only if $Q$ intersects both
$I_n$ and $J_n$, and will reach its maximum when $N\approx n$.  But in
this case we can 
replace $Q$ by $(e^n,e^n+n)$ and the above computation shows
that $\|u\|_{\Phi,Q}\|\sigma\|_{1,Q}\lesssim 1$.

Finally, suppose $Q$ intersects one or more pairs $I_n$ and $J_n$ with
$n \lesssim \log(N)$.   Then $|\supp(u)\cap Q|\lesssim \log(N)$ and 
$\|u\|_{L^\infty(Q)} \approx K_{\lfloor\log(N)\rfloor} \lesssim \log(N)^2$.  Therefore,
\[ \|u\|_{\Phi,Q} \lesssim \|u\|_{2,Q} \leq 
\|u\|_{L^\infty(Q)} \left(\frac{|\supp(u)\cap Q|}{|Q|}\right)^{1/2}
\lesssim \frac{\log(N)^{5/2}}{N^{1/2}}. \]
A similar calculation shows that
\[ \|\sigma \|_{1,Q} \lesssim \frac{\log(N)}{N}; \]
hence, we again have that  $\|u\|_{\Phi,Q}\|\sigma\|_{1,Q}\lesssim 1$.
It follows that $[u,\sigma]_{A,2}<\infty$ and our
proof is complete.

\begin{remark}l
If we modify our example by defining $K_n= n^2\log(e+n)^{-2}$, then
the same argument shows that $(u,\sigma)$ satisfy the separated bump
condition when $A(t)=B(t)=t^2\log(e+t)^{1+\delta}$, $0<\delta<2$, but
do not satisfy the double bump condition for any $\delta>0$.  It would
be of interest to construct a pair that satisfies a separated bump
condition for some pair of log bumps but fails to satisfy the double
bump condition for any pair of Young functions for which the
appropriate $B_p$ conditions hold.
\end{remark}

\bibliographystyle{plain}
\bibliography{SHT-bump}

\end{document}